\documentclass{article}

\usepackage{arxiv}

\usepackage[utf8]{inputenc} 
\usepackage[T1]{fontenc}    
\usepackage{hyperref}       
\usepackage{url}            
\usepackage{booktabs}       
\usepackage{amsfonts}       
\usepackage{nicefrac}       
\usepackage{microtype}      
\usepackage{graphicx}
\usepackage{doi}
\usepackage{physics,amsmath}
\usepackage{authblk}
\usepackage{amsthm}
\usepackage{placeins}
\usepackage{soul}
\usepackage{xcolor}
\usepackage{float}
\usepackage{bm}
\usepackage{todonotes}
  \usepackage[
    backend=biber,
    style=numeric,
  ]{biblatex}

\newtheorem{theorem}{Theorem}[section]
\newtheorem{lemma}[theorem]{Lemma}

\newcommand{\jump}[1]{\ensuremath{[\![#1]\!]} }

\newcommand{\dint}[1]{\left\langle #1 \right\rangle}
\newcommand{\pderiv}[2]{\frac{\partial #1}{\partial #2}}

\newcommand{\G}[1]{ 
\dint{\phi, \pderiv{#1}{t} + \boldsymbol{u} \cdot \nabla #1}_{\Omega_k\times T^n} - \dint{\phi, \left(q^{n-1} - #1\right)}_{\Omega_k}\bigg|_{t^{n-1}}}

\title{Scalable ADER-DG Transport Method with Polynomial Order Independent CFL Limit}

\author[1,2]{Kieran Ricardo}
\author[1,3]{Kenneth Duru}
\affil[1]{Mathematical Sciences Institute, Australian National University, Canberra, Australia}
\affil[2]{ACCESS-NRI, Canberra, Australia}
\affil[3]{Department of Mathematical Sciences, University of Texas at El Paso, USA}

\date{}

\graphicspath{ {./plots} }

\begin{document}
\maketitle

\begin{abstract}

Discontinuous Galerkin (DG) methods are known to suffer from increasingly restrictive explicit time-step constraints as the polynomial order increases, limiting their efficiency at high orders for explicit time-stepping schemes. In this paper, we introduce a novel \emph{locally implicit}, but \emph{globally explicit} ADER-DG scheme designed for transport-dominated problems. The method achieves a maximum stable time step governed by an element-width based CFL condition that is independent of the polynomial degree. By solving a set of element-local implicit problems at each time step, our approach more effectively utilises the domain of dependence. As a result, our method remains stable for CFL numbers up to $\approx 1/\sqrt{d}$ in $d$ spatial dimensions. We provide a rigorous stability proof in one dimension, and extend the analysis to two and three dimensions using a semi-analytical von Neumann stability analysis. The accuracy and convergence of the method are demonstrated through numerical experiments for both linear and nonlinear test cases, including numerical simulations of a transport problem on a  cubed sphere 2D manifold embedded in 3D.
\end{abstract}


\section{Introduction}
Transport models are a fundamental component of many fluid dynamics solvers. In geophysical fluid dynamics, models may include dozens of advected or transported quantities, motivating the need for accurate, efficient, and scalable transport algorithms \cite{BRADLEY2024113034,LauritzenThurbon2012,gmd-8-1299-2015,gmd-15-6285-2022}. Discontinuous Galerkin (DG) methods have recently gained popularity for modelling geophysical flows due to their scalability, high-order accuracy, and strong conservation properties \cite{kelly2012continuous, marras2016review, dennis2005towards, waruszewski2022entropy, ricardo2024thermodynamic}. However, a well-known limitation of high-order DG methods, common to many finite element schemes, is that the maximum stable explicit time step decreases quadratically with the polynomial degree \cite{karniadakis2005spectral}. This constraint may render high-order methods impractical for large-scale or real-time simulations.

One approach to circumvent this time-step restriction is to use fully implicit time integration, but such methods typically involve global linear or nonlinear solves and increased communication overhead per time step, which can limit scalability on modern high-performance computing systems. Our focus is on the development of communication-avoiding schemes that remain efficient at scale. In this work, we introduce a novel \emph{locally implicit}, but \emph{globally explicit} ADER-DG method that removes the dependence of the maximum stable time step on polynomial order while minimizing inter-element communication.

The DG method was originally developed for neutron transport problems \cite{reed1973triangular} and later extended to a broader class of hyperbolic conservation laws and non-conservative hyperbolic partial differential equations \cite{cockburn1989tvb, cockburn1990runge, cockburn1991runge, cockburn1998runge, cockburn1989tvb1}. A widely used strategy for solving time-dependent hyperbolic problems with DG is the method of lines, in which DG is used for spatial discretization and a strong stability preserving Runge–Kutta (SSP-RK) method is applied for time integration \cite{shu1988efficient}.

Typically, a DG method with polynomial degree $k-1$—achieving order $k$ accuracy in space—is combined with an SSP-RK scheme of order $k$ to obtain overall $k$th-order accuracy in both space and time. In SSP-RK methods, elements communicate with their neighbors via numerical fluxes at each Runge–Kutta stage.
However, a fundamental limitation arises from the fact that $k$-stage, $k$th-order explicit Runge–Kutta methods do not exist for $k > 4$ \cite{butcher1964runge}. At each Runge-Kutta stage elements this limits the efficiency of the method-of-lines approach for DG methods with polynomial degree $> 3$.

The ADER-DG, or locally implicit DG, method overcomes this limitation by using a DG discretisation in both space and time \cite{toro2001godunov, toro2002solution,titarev2002ader,titarev2005ader, toro2006derivative, dumbser2008finite, dumbser2008unified, gassner2011explicit,DURU2022114386}. Regardless of the polynomial order or the problem’s dimensionality, the ADER-DG method requires only one round of communication across element interfaces. Due to its arbitrary order of accuracy and low communication cost, the ADER-DG method has been successfully applied to a wide range of problems \cite{fambri2017space, gaburro2023high, pelties2014verification, breuer2016petascale, breuer2022next, rannabauer2018ader, popov2025high}.

For all explicit schemes applied to hyperbolic problems, the maximum stable time step is constrained by the Courant–Friedrichs–Lewy (CFL) condition \cite{courant1928partiellen}. This condition ensures that the numerical domain of dependence includes the physical domain of dependence, thereby maintaining stability. In 1D the CFL condition can be expressed as
$$\nu := \frac{c\Delta t}{\Delta x} \leq 1,$$
where $\Delta t$ is the time step, $\Delta x$ is grid spacing, $c$ is the fastest wave speed, and $\nu$ is the CFL number. In practice, both explicit RK-DG and ADER-DG methods have maximum stable time steps that are well below the CFL limit \cite{liu2008l2, guthrey2019regionally}. Table 1 and 2 illustrate the maximum stable CFL number for explicit Runge-Kutta DG and standard ADER-DG schemes in 1D, showing that it decreases sharply with polynomial order, making high order methods computationally inefficient.


\begin{table}
\caption{1D Runge-Kutta DG CFL condition numbers \cite{liu2008l2}}
\centering
\begin{tabular}{||l||l|l|l|l|l|l|}
\hline
Polynomial order     & 0   & 1    & 2    & 3   \\ \hline
CFL cond. & 1.0 & 0.33 & 0.13 & 0.1 \\ \hline
\end{tabular}
\end{table}

\begin{table}
\caption{1D ADER-DG CFL condition numbers \cite{guthrey2019regionally}}
\centering
\begin{tabular}{||l||l|l|l|l|l|l|}
\hline
Polynomial order     & 0   & 1    & 2    & 3  & 4 & 5\\ \hline
CFL cond.   & 1.0 & 0.33 & 0.17 & 0.1 & 0.07 & 0.05\\ \hline
\end{tabular}
\end{table}

A recent approach to mitigating this limitation is the regionally implicit DG method \cite{guthrey2019regionally}. This method achieves a CFL number of 1, independent of the polynomial order, while requiring only two communication steps per time step. At each time step, and for each element, an implicit problem is solved over a small region consisting of all neighbouring elements, specifically, $3^d$ elements in spatial dimension $d$, to compute a local predictor. The standard ADER-DG corrector step is then applied. Although this approach significantly increases the maximum stable time step compared to the {standard ADER-DG method and maintains low communication overhead, the resulting implicit problem may be relatively large and computationally expensive to solve.

In this paper, we propose an alternative locally implicit ADER-DG method for transport problems, designed to improve efficiency on high-performance computing systems. Unlike the regionally implicit approach, our method solves multiple single-element implicit problems per time step. In $d$ spatial dimensions, this involves solving $2^d$ single-element implicit problems per element and requires $d$ rounds of communication between face-sharing neighbors. Although the method results in a lower maximum stable CFL number, $1 / {\sqrt{d}}$ in $d$ spatial dimensions, compared to the regionally implicit DG method each implicit problem is element local and can be solved efficiently using standard iterative ADER-DG techniques. This smaller problem size may lead to better overall efficiency for certain applications. 

The rest of the paper is organized as follows: Section 2 introduces the notation and the standard ADER-DG method. Section 3 presents our novel locally implicit DG method. In Section 4, we prove our 1D method is numerically stable for element-size CFL numbers $\leq 1$. Section 5 conducts a numerical von Neumann analysis to determine the maximum stable time step of our method in 1D, 2D, and 3D. Section 6 presents the results of our numerical experiments, where we verify the convergence of our method, including convergence tests for the advection equation on a  cubed sphere 2D manifold embedded in 3D. Finally, our conclusions from this study are presented in
section 7.

\section{Space-time discontinuous Galerkin method}

In this section, we introduce the fully implicit space-time discontinuous Galerkin method \cite{dumbser2016space} and the standard locally implicit ADER-DG method \cite{dumbser2008finite, dumbser2008unified, gassner2011explicit} for the linear advection equation in a periodic domain $\Omega$
\begin{equation}
    \pderiv{q}{t} + \boldsymbol{u}\cdot \nabla q = 0,
\end{equation}
with the smooth and periodic initial condition $q(x, 0) = f(x)$.
Here, $q$ represents the advected quantity and 
$\boldsymbol{u}$$\in \mathbb{R}^d$ is the advection velocity vector. 

\subsection{Implicit space-time discontinuous Galerkin method}

The computational domain is discretised into spatial elements $\Omega_k$, which are then extended into space-time elements $\Omega_k \cross T^n$, where $T^n = [t^{n-1}, t^n]$ denotes the $nth$ time interval. Within each space-time element, the solution 
$q_k^n(\boldsymbol{x}, t)$ is approximated by a finite polynomial expansion in terms of basis functions:
\begin{equation}
q_k^n=\sum_i{\left(q_k^n\right)_i \phi_i(\boldsymbol{x}, t)},
\end{equation}
where $\phi_i(\boldsymbol{x}, t)$ are typically chosen as tensor-product polynomial basis functions in space and time. Note that we will drop the subscript $k$ when it is clear from context.

The implicit space-time DG method \cite{dumbser2016space} follows the standard DG framework, extended to the space-time setting. The governing partial differential equation is multiplied by each basis function $\phi_i$,  integrated over the space-time element $\Omega_k \cross T^n$, and coupled to neighbouring elements via numerical fluxes. This leads to the weak form:
{
\scriptsize
\begin{equation}
    \dint{\phi, \pderiv{q^n}{t} +  \boldsymbol{u}\cdot \nabla q^n}_{\Omega_k\cross T^n} + \dint{\phi,  \left(\hat{q}^n-q^n\right)\boldsymbol{u}\cdot \boldsymbol{n}}_{\partial\Omega_k \cross T^n} -  \dint{\phi, \left(q^{n-1} -  q^n\right)}_{\Omega_k}\bigg|_{t^{n-1}} = 0,
    \label{eq:imp-ader}
\end{equation}
}
where $\dint{\cdot, \cdot}$ denotes discrete inner products where we have approximated the integrals by a Gauss-type quadrature, and $\hat{q}^n$ denotes the upwind value of $q^n$ at the element interface. Note that upwind numerical fluxes are employed both at spatial interfaces and at the temporal boundary $t = t^{n-1}$, ensuring the discretisation is stable and consistent for arbitrarily large time-steps. 

\begin{theorem}
    The implicit ADER-DG method \eqref{eq:imp-ader} with constant advection velocity vector $\boldsymbol{u}$ is unconditionally stable, that is 
    \begin{equation}
    \begin{split}
        &\sum_k \dint{\tfrac{1}{2}\left(q^n\right)^2}_{\Omega_k}\bigg|_{t^n} -\dint{\tfrac{1}{2}\left(q^{n-1}\right)^2}_{\Omega_k}\bigg|_{t^{n-1}} = \\ -&\sum_k \left(\dint{\tfrac{1}{4}|\boldsymbol{u}\cdot \boldsymbol{n}|\jump{q^n}^2}_{\partial\Omega_k \cross T^n} + \dint{\tfrac{1}{2}\left(q^n - q^{n-1}\right)^2}_{\Omega_k }\bigg|_{t^{n-1}}\right)  \leq 0.
        \end{split}
        \end{equation}
\end{theorem}
Although this scheme is unconditionally stable, it requires solving a global implicit problem at every time step, which increases computational complexity and communication overhead compared to explicit methods and can limit scalability on modern high-performance computing systems. An  alternative is the standard explicit ADER-DG method, which mitigates this issue by requiring an element locally implicit solve at the predictor stage and a global explicit solve at the predictor stage, consequently leads to only a single communication round per time step.

\subsection{Standard ADER discontinous Galerkin method}

The key idea of the ADER-DG method \cite{dumbser2008finite, dumbser2008unified, gassner2011explicit} is to solve a locally implicit problem in each element to obtain a predictor $\tilde{q}$, and then use this predictor to calculate the spatial numerical fluxes and all other spatial terms in \eqref{eq:imp-ader}.

The element local predictor $\tilde{q}$ is given by
\begin{equation}
	\dint{\phi, \pderiv{\tilde{q}}{t} + \boldsymbol{u} \cdot \nabla \tilde{q}}_{\Omega_k\times T^n} - \dint{\phi, \left(q^{n-1} - \tilde{q}\right)}_{\Omega_k}\bigg|_{t^{n-1}} = 0.
    \label{eq:exp-ADER-pred}
\end{equation}
This is then used to calculate all spatial terms in \eqref{eq:imp-ader}, giving $q^n$ as
\begin{equation}
    \dint{\phi, \pderiv{q^n}{t} + \boldsymbol{u}\cdot \nabla \tilde{q}}_{\Omega_k\cross T^n} + \dint{\phi,  \left(\hat{\tilde{q}}-\tilde{q}\right)\boldsymbol{u}\cdot \boldsymbol{n}}_{\partial\Omega_k \cross T^n} -  \dint{\phi, \left(q^n - q^{n-1}\right)}_{\Omega_k}|_{t=t^n} = 0.
    \label{eq:exp-ader}
\end{equation}

The element local implicit problem $\eqref{eq:exp-ADER-pred}$ is solved iteratively by 
\begin{equation}
	\dint{\phi, \pderiv{\tilde{q}^{i+1}}{t}}_{\Omega_k\times T^n} + \dint{\phi, \tilde{q}^{i+1}}_{\Omega_k}\bigg|_{t^{n-1}} = -\dint{\phi,\boldsymbol{u} \cdot \nabla \tilde{q}^i}_{\Omega_k\times T^n} + \dint{\phi, q^{n-1}}_{\Omega_k}\bigg|_{t^{n-1}},
\end{equation}
where the superscript $i$ denotes the $i$-th iteration. This iterative procedure also converges for non-linear problems. In the next section, we will use a similar iterative approach for all locally implicit problems in our method.

Although the ADER-DG method requires only a single communication round per time step and therefore scales well on modern high-performance computing systems, its maximum stable time step decreases quadratically with polynomial order, reducing the efficiency of high-order methods. Our approach seeks to eliminate this time step restriction while ensuring scalability.

\section{Novel locally implicit ADER-DG method}

In this section, we introduce our locally implicit ADER-DG method for 1D, 2D, and 3D transport problems. Compared to the standard locally implicit DG method, our scheme allows for larger time steps, independent of the polynomial order, minimizing the number of communication steps required for higher-order methods.

\subsection{Notation}

In two spatial dimensions, we use quadrilateral elements; in three dimensions, we use hexahedral elements. For both space and time, we employ tensor-product Gauss–Lobatto–Legendre (GLL) quadrature, along with tensor-product Lagrange polynomial basis functions collocated with the GLL points.

We define the reference spatial element as $\tilde{\Omega} = [-1, 1]^d$ in $d$ spatial dimensions, with reference coordinates denoted by $\boldsymbol{\xi} = (\xi_1, \dots, \xi_d)$. Each physical element $\Omega_k$ is the image of an invertible, element-specific mapping $\boldsymbol{r}_k(\boldsymbol{\xi}) : \tilde{\Omega} \to \Omega_k$ from the reference element to the physical domain.

In dimensions $d=2$ and $d=3$, we must explicitly identify the portions of the element boundary over which integrals are evaluated. For this purpose, we define the element face corresponding to the boundary where $\xi_i = \pm 1$ as:
$$\partial \Omega_k^{\xi_i} = \{\boldsymbol{r}_k(\boldsymbol{\xi}) : \xi_i = \pm 1 \}.$$

To incorporate the time dimension, we define the reference time interval as $\tilde{T} = [-1, 1]$ and denote the reference time coordinate by $\tau \in \tilde{T}$. The reference space-time element is then given by the Cartesian product $\tilde{\Omega} \times \tilde{T}$. Integration in time is carried out using the same GLL quadrature strategy as in space.


\subsection{1D}

In 1D, our method employs the same predictor as the standard locally implicit DG method \eqref{eq:exp-ADER-pred}
\begin{equation}
	\G{\tilde{q}} = 0.
\end{equation}

During the corrector phase, the predictor is used solely to compute the numerical fluxes, while all other terms are treated implicitly
{
\scriptsize
\begin{equation}
	\G{q^n} =- \dint{\phi,\boldsymbol{u} \cdot \boldsymbol{n}\left(\hat{\tilde{q}}   - q^{n}\right)}_{\partial \Omega_k\times T^n}.
\end{equation}
}
This approach differs from the standard ADER-DG corrector, which uses the predictor in both the volume and flux terms. Intuitively, for time steps below the element CFL limit, the fluxes at the interface of two elements should depend only on the two elements themselves, since information from the wider domain has not yet propagated to the interface. This allows us to construct accurate and stable fluxes for an element CFL of up to 1.


\subsection{2D}

A straightforward extension of our 1D method to 2D is numerically unstable. Intuitively, this instability arises because the fluxes at the
$x$-interface now depend on neighbouring elements in the $y$-direction. To address this issue, we adopt a two-step procedure to obtain numerical fluxes for the corrector: first, we apply our 1D method in the $x$-direction to obtain fluxes for $y$ interfaces; second, we apply the same method in the $y$-direction to obtain fluxes at the $x$ interfaces. These $x$-flux and $y$-flux predictors are independent and depend only on the initial predictor, therefore our 2D method has 2 communication stages and 4 element local implicit solves.  Figure \ref{fig:algorithm-2D} summarises the overall 2D algorithm.

\begin{figure}[!hbtp]\begin{center}
	\includegraphics[width=0.5\textwidth]{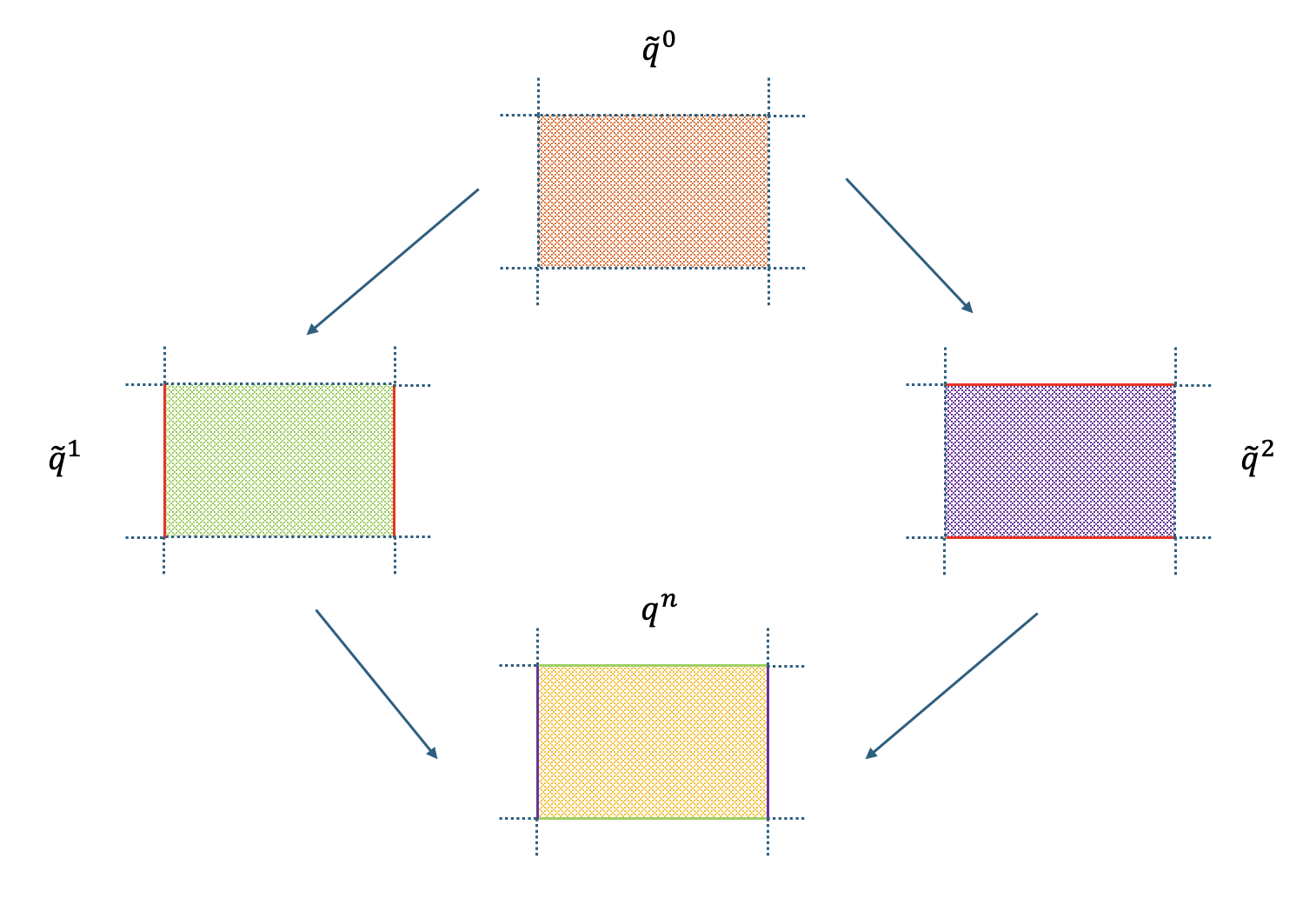}
 \caption{New locally implicit ADER-DG method in two spatial dimensions. A dotted boundary indicates no numerical flux, and coloured solid boundaries denote upwind numerical fluxes calculated with the predictor of the corresponding colour.}
  \label{fig:algorithm-2D}
\end{center}
\end{figure}

The first predictor is the same as in the 1D case
\begin{equation}
    \G{\tilde{q}^0} = 0.
\end{equation}
The 2D method introduces introduces 2 new predictors, $\tilde{q}^1$ and $\tilde{q}^2$, which use numerical fluxes from the first predictor $\tilde{q}^0$ along the $\xi_1=\pm 1$ and $\xi_2=\pm 1$ boundaries respectively. These predictors are given by
{
\scriptsize
\begin{equation}
    \G{\tilde{q}^1} =  - \dint{\phi,\boldsymbol{u} \cdot \boldsymbol{n}\left(\hat{\tilde{q}}^0   - \tilde{q}^1\right)}_{\partial \Omega_k^{\xi_1}\times T^n},
\end{equation}
}
{
\scriptsize
\begin{equation}
    \G{\tilde{q}^2} =  - \dint{\phi,\boldsymbol{u} \cdot \boldsymbol{n}\left(\hat{\tilde{q}}^0   - \tilde{q}^2\right)}_{\partial \Omega_k^{\xi_2}\times T^n}.
\end{equation}
}
Finally, the corrector step is computed as
{
\scriptsize
\begin{equation}
\begin{split}
    \G{q^n}= &- \dint{\phi_h,\boldsymbol{u} \cdot \boldsymbol{n}\left(\hat{\tilde{q}}^2   - q^{n}\right)}_{\partial \Omega_k^{\xi_1}\times T}\\
    &- \dint{\phi,\boldsymbol{u} \cdot \boldsymbol{n}\left(\hat{\tilde{q}}^1   - q^{n}\right)}_{\partial \Omega_k^{\xi_2}\times T^n}.
    \end{split}
\end{equation}
}
\subsection{3D}
In 3D, the $x$ fluxes now depend on neighbouring elements in the $y$ and $z$ directions. To address this, we apply the 2D method in the $y$- and $z$-directions to obtain $x$-fluxes for the final corrector, and similarly use the same approach for calculating the $y$- and $z$-fluxes. Since some predictors are reused multiple times, the method only requires 3 communication steps and 8 element local implicit solves. 

The first four predictors follow the same form as the 2D predictors
\begin{equation}
    \G{\tilde{q}^0} = 0,
\end{equation}
{
\scriptsize
\begin{equation}
	\G{\tilde{q}^i} =  - \dint{\phi,\boldsymbol{u} \cdot \boldsymbol{n}\left(\hat{\tilde{q}}^0   - \tilde{q}^i\right)}_{\partial \Omega_k^{\xi_i}\times T^n},
\end{equation}
}
for $i=1,2,3$. To obtain the final three predictors, $\tilde{q}^{i,j}$ for $i,j=(1,2),(1,3),(2,3)$, we perform the final 2D corrector step giving
{
\scriptsize
\begin{equation}
\begin{split}
	 \G{\tilde{q}^{i,j}} =  &- \dint{\phi,\boldsymbol{u} \cdot \boldsymbol{n}\left(\hat{\tilde{q}}^j   - \tilde{q}^{i,j}\right)}_{\partial \Omega_k^{\xi_i}\times T^n} \\
     &- \dint{\phi,\boldsymbol{u} \cdot \boldsymbol{n}\left(\hat{\tilde{q}}^i   - \tilde{q}^{i,j}\right)}_{\partial \Omega_k^{\xi_j}\times T^n}.
\end{split}
\end{equation}
}
For example, $\tilde{q}^{1,2}$ is equivalent to performing the 2D method in the $\xi_1$ and $\xi_2$ directions, yielding stable and high order accurate fluxes at the $\xi_3 = \pm 1$ boundaries. With these predictors in place, we can construct stable approximations for all element interface fluxes and compute the final corrector as
{
\scriptsize
\begin{equation}
\begin{split}
\G{q^n_k} = & - \dint{\phi,\boldsymbol{u}_h \cdot \boldsymbol{n}\left(\hat{\tilde{q}}^{2,3}  - q^{n+1}_h\right)}_{\partial \Omega_k^{\xi_1}\times T}  \\  &- \dint{\phi,\boldsymbol{u} \cdot \boldsymbol{n}\left(\hat{\tilde{q}}^{1,3}   - q^{n+1}\right)}_{\partial \Omega_k^{\xi_2}\times T} \\  &- \dint{\phi_h,\boldsymbol{u} \cdot \boldsymbol{n}\left(\hat{\tilde{q}}^{1,2}   - q^{n+1}\right)}_{\partial \Omega_k^{\xi_3}\times T^n}.
   \end{split}
\end{equation}
}

\section{1D Stability Analysis}

In this section, we present theoretical results that demonstrate the stability of our 1D method for a given constant CFL number and polynomial order. Specifically, for order $p$ we derive a $(p+1)\times (p+1)$ eigenvalue problem and prove that a maximum eigenvalue of 1 implies stability. We evaluate stability for polynomial orders up to $p \leq 14$ and across a range of CFL numbers, finding that our method remains stable for CFL values up to 1.

While these results are based on stability evaluations for specific CFL numbers, and thus only confirm stability for the tested values, we make the reasonable assumption that the maximum amplification factor increases with the CFL number. Based on this assumption, we conclude that our 1D method is stable for CFL values up to 1.0 for polynomial orders $\leq 14$. Although we suspect that higher-order methods are also stable, achieving conclusive results may require higher precision in the numerical linear algebra due to poorly conditioned matrices for high order operators.

For the proofs we introduce additional notation. Let $x_k$ denote the left boundary point of cell $k$, s.t. $\Omega_k = [x_k, x_{k+1}]$. This allows us to distinguish between integrals along the left an right element boundaries. For example, $\dint{a,b}_{\tilde{T}}|_{x_k}$ approximates the left time boundary integral $\int_{\tau=-1}^{\tau=1}{a(x=x_k)b(x=x_k)d\tau}$. 

We also transform the 1D method into the reference element, and scale out a factor of $\Delta x$. Now the geometry dependence is entirely contained within the CFL number $\nu$. In reference coordinates our 1D method becomes
\begin{equation}
    \dint{\phi, \pderiv{\tilde{q}_k}{\tau}}_{\tilde{\Omega} \times \tilde{T}} + \dint{\phi, \nu \pderiv{\tilde{q}_k}{\xi}}_{\tilde{\Omega} \times \tilde{T}} - \dint{\phi, q^{n-1}_k - \tilde{q}_k}_{\tilde{\Omega}}\bigg|_{t^{n-1}} = 0,
    \label{eq:ref-1D-1}
\end{equation}
\begin{equation}
    \dint{\phi, \pderiv{q^{n}_k}{\tau}}_{\tilde{\Omega} \times \tilde{T}} + \dint{\phi, \nu \pderiv{q^{n}_k}{\xi}}_{\tilde{\Omega} \times \tilde{T}} - \dint{\phi, q^{n-1}_k - q^{n}_k}_{\tilde{\Omega}}\bigg|^{t^n} + \dint{\phi, \nu \left(\hat{\tilde{q}} - q^{n}_k\right)}_{\tilde{T}}\bigg|^{x_{k+1}}_{x_k} = 0.
    \label{eq:ref-1D-2}
\end{equation}
We also assume without loss of generality that $u \geq 0$. Consequently, an upwind value $\hat{\tilde{q}}|_{x_k}=\tilde{q}_{k-1}|_{x_k}$ represents the value on the left side of the element boundary, therefore \eqref{eq:ref-1D-2} can be written as
\begin{equation}
\begin{split}
    \dint{\phi, \pderiv{q^{n}_k}{\tau}}_{\tilde{\Omega} \times \tilde{T}} + \dint{\phi, \nu \pderiv{q^{n}_k}{\xi}}_{\tilde{\Omega} \times \tilde{T}} - &\dint{\phi, q^{n-1}_k - q^{n+1}_k}_{\tilde{\Omega}}\bigg|_{t^{n-1}} + \dint{\phi, \nu \left(\tilde{q}_k - q^{n}_k\right)}_{\tilde{T}}\bigg|_{x_{k+1}} \\
    &- \dint{\phi, \nu \left(\tilde{q}_{k-1} - q^{n}_k\right)}_{\tilde{T}}\bigg|_{x_k}= 0.
    \label{eq:ref-1D-3}
    \end{split}
\end{equation}

\subsection{Analysis}

We begin by expressing the energy change after one step of the 1D method as a function of two key quantities: the difference between the corrector and predictor at time $t^n$, and the jump in the predictor across spatial element boundaries. This formulation enables the construction of a stability eigenvalue problem. This result is presented in Theorem \ref{theorem:energy-1-step}. To derive it, we first establish two preliminary lemmas.

We start by showing that the predictor recovers the solution from the previous time step at $t^{n-1}$. This result is formalised in the following lemma.
\begin{lemma}
    At the beginning of the time step, $t^{n-1}$, the predictor $\tilde{q}_k$ equals the previous solution $q^{n-1}$, that is $\tilde{q}_k(\xi, \tau=-1)=q^{n-1}_k(\xi, \tau=1)$.
\end{lemma}
\begin{proof}
    Let $\xi_i$ be the GLL quadtrature points and $l_i$ be the Lagrange polynomials which interpolate these points. The function \begin{equation}
        \tilde{q}_k = \sum_{i}{q^{n-1}_k\left(\xi=\xi_i, \tau=1\right)l_i\left(\xi  -\nu(\tau + 1)\right)}
    \end{equation} satisfies both 
    \begin{equation}
        \dint{\phi, \pderiv{\tilde{q}_k}{t} + \nu\pderiv{\tilde{q}_k}{x}}_{\tilde{\Omega} \times \tilde{T}} = 0
    \end{equation}
    and 
    \begin{equation}
    \tilde{q}_k(\xi, \tau=-1)=\sum_i{q^{n-1}_k\left(\xi=\xi_i, \tau=1\right)l_i\left(\xi\right)} = q^{n-1}_k(\xi, \tau=1),
    \end{equation}
    and therefore also satisfies 
    \begin{equation}
        \dint{\phi, \pderiv{\tilde{q}_k}{t} + \nu\pderiv{\tilde{q}_k}{x}}_{\tilde{\Omega} \times \tilde{T}} - \dint{\phi, q^{n-1}_k - \tilde{q}_k}_{\tilde{\Omega}}\bigg|_{t^{n-1}} = 0.
    \end{equation}
\end{proof}

Lemma 4.1 allows us to formulate the evolution of the difference between the predictor and the corrector, $\psi_k^n=q^{n}_k - \tilde{q}_k$, almost entirely in terms of $\psi_k^n$. Taking the difference between \eqref{eq:ref-1D-3} and \eqref{eq:ref-1D-1}, and applying Lemma 4.1 yields
{
\scriptsize
\begin{equation}
    \dint{\phi, \pderiv{\psi_k^n}{\tau}}_{\tilde{\Omega} \times \tilde{T}} + \dint{\phi, \nu \pderiv{\psi_k^n}{\xi}}_{\tilde{\Omega} \times \tilde{T}} + \dint{\phi, \psi_k^n}_{\tilde{\Omega}}\bigg|_{t^{n-1}} - \dint{\phi, \nu \psi_k^n}_{\tilde{T}}\bigg|_{x_{k+1}} = \dint{\phi, \nu\left(\tilde{q}_{k-1} - q^{n}_{k}\right)}_{\tilde{T}}\bigg|_{x_k}.
    \label{eq:pred-corr-diff-1D-1}
\end{equation}
}
This can also be written as
{
\scriptsize
\begin{equation}
    \dint{\phi, \pderiv{\psi_k^n}{\tau}}_{\tilde{\Omega} \times \tilde{T}} + \dint{\phi, \nu \pderiv{\psi_k^n}{\xi}}_{\tilde{\Omega} \times \tilde{T}} + \dint{\phi, \psi_k^n}_{\tilde{\Omega}}\bigg|_{t^{n-1}} - \dint{\phi, \nu \psi_k^n}_{\tilde{T}}\bigg|^{x_{k+1}}_{x_k} = \dint{\phi, \nu\left(\tilde{q}_{k-1} - \tilde{q}_{k}\right)}_{\tilde{T}}\bigg|_{x_k}.
    \label{eq:pred-corr-diff-1D-2}
\end{equation}
}

Various boundary integrals of $\left(\psi_k^n\right)^2$ arise in our analysis. The following lemma establishes a useful relationship between these boundary integrals.
\begin{lemma} Consider the evolution equation defined by \eqref{eq:pred-corr-diff-1D-1}. We have
{
\scriptsize
\begin{equation}
\begin{split}
    \dint{\tfrac{1}{2}\nu \left(\psi_k^n\right)^2}_{\tilde{T}}\bigg|_{x_{k+1}} - \dint{\tfrac{1}{2}\left(\psi_k^n\right)^2}_{\tilde{\Omega}}\bigg|_{t^{n-1}} &= \dint{\tfrac{1}{2}\left(\psi_k^n\right)^2}_{\tilde{\Omega}}\bigg|_{t^{n}}  - \dint{\psi_k^n, \nu\left(\tilde{q}_{k-1} - q^{n}_k\right)}_{\tilde{\Omega}}\bigg|_{x_k} \\
    &- \dint{\tfrac{1}{2}\nu\left(\psi_k^n\right)^2}_{\tilde{\Omega}}\bigg|_{x_k}
    \end{split}
\end{equation}
}
\end{lemma}
\begin{proof}
    Substituting $\phi=\psi^n_k$ in \eqref{eq:pred-corr-diff-1D-1} and using SBP in both space and time yields the desired result.
\end{proof}

With these preliminary results, we are now ready to prove the following theorem, which states that the energy growth of the 1D method depends only on the difference between the corrector and the predictor, $\psi$, and the jumps in the predictor $\tilde{q}$. When combined with \eqref{eq:pred-corr-diff-1D-2}, this enables us to establish numerical stability for a given polynomial order and CFL number.
\begin{theorem} \label{theorem:energy-1-step}
The energy change after one step of the 1D method,
{
\scriptsize
\eqref{eq:ref-1D-1}-\eqref{eq:ref-1D-2}, is
    \begin{equation}
        \Delta E^n = \sum_k{\dint{\tfrac{1}{2}\left(q^{n}_k\right)^2}_{\tilde{\Omega}}\bigg|_{t^{n}} - \dint{\tfrac{1}{2}\left(q^{n-1}_k\right)^2}_{\tilde{\Omega}}\bigg|_{t^{n-1}}} = \sum_k{
        \dint{\tfrac{1}{2}\left(\psi^n_k\right)^2}_{\tilde{\Omega}}\bigg|_{t^{n}} - \dint{\tfrac{1}{2}\nu\left(\tilde{q}_{k-1} - \tilde{q}_k\right)^2}_{\tilde{T}}\bigg|_{x_k}}.
        \label{eq:1D-energy-change}
    \end{equation}
    }

\end{theorem}
\begin{proof}
    Substituting $\phi=q^{n}_k$ into \eqref{eq:ref-1D-2} and using SBP in both space and time yields
    \begin{equation}
        \dint{\tfrac{1}{2}\left(q^{n}_k\right)^2}_{\tilde{\Omega}}\bigg|_{t^{n}} = \dint{\nu\left(\left(\tfrac{1}{2}q^{n}_k\right)^2 - q^{n}_k\hat{\tilde{q}}\right)}_{\tilde{T}}\bigg|_{x_k}^{x_{k+1}} - \dint{\tfrac{1}{2}\left(q^{n}_k\right)^2 - q^{n}_kq^{n-1}_k}_{\tilde{\Omega}}\bigg|_{t^{n-1}}, 
    \end{equation}
    therefore
    {
\scriptsize
    \begin{equation}
        \dint{\tfrac{1}{2}\left(q^{n}_k\right)^2}_{\tilde{\Omega}}\bigg|_{t^{n}} - \dint{\tfrac{1}{2}\left(q^{n-1}_k\right)^2}_{\tilde{\Omega}}\bigg|_{t^{n-1}} = \dint{\nu\left(\left(\tfrac{1}{2}q^{n}_k\right)^2 - q^{n}_k\hat{\tilde{q}}_k\right)}_{\tilde{T}}\bigg|_{x_k}^{x_{k+1}} - \dint{\tfrac{1}{2}\nu\left(q^{n}_k - q^{n-1}_k\right)^2}_{\tilde{\Omega}}\bigg|_{t^{n-1}}. 
    \end{equation}
    }
    Noting that as $\hat{\tilde{q}}$ is continuous across element boundaries, in a periodic domain $\sum_k{\dint{\tfrac{1}{2}\nu\hat{\tilde{q}}_k^2}_{\tilde{T}}\bigg|_{x_k}^{x_{k+1}}} = 0$, therefore
    {
\scriptsize
    \begin{equation}
        \Delta E^n = \Delta E^n + \sum_k{\dint{\tfrac{1}{2}\nu\hat{\tilde{q}}_k^2}_{\tilde{T}}\bigg|_{x_k}^{x_{k+1}}} = \sum_k{\dint{\tfrac{1}{2}\nu\left(q^{n}_k - \hat{\tilde{q}}_k\right)^2}_{\tilde{T}}\bigg|_{x_k}^{x_{k+1}} - \dint{\tfrac{1}{2}\left(q^{n}_k - q^{n-1}_k\right)^2}_{\tilde{\Omega}}\bigg|_{t^{n-1}}}.
    \end{equation}
    }
We apply Lemma 4.1, substituting $q^{n-1}|_{t^{n-1}} = \tilde{q}|_{t^{n-1}}$, and also use $\hat{\tilde{q}}|_{x_k} = \tilde{q}_{k-1}|_{x_k}$ to obtain
{
\scriptsize
\begin{equation}
    \Delta E^n = \sum_k{\dint{\tfrac{1}{2}\nu\left(q^{n}_k - \tilde{q}_k\right)^2}_{\tilde{T}}\bigg|_{x_{k+1}} - \dint{\tfrac{1}{2}\nu\left(q^{n}_k - \tilde{q}_{k-1}\right)^2}_{\tilde{T}}\bigg|_{x_k} - \dint{\tfrac{1}{2}\left(q^{n}_k - \tilde{q}_k\right)^2}_{\tilde{\Omega}}\bigg|_{t^{n-1}}}.
\end{equation}
}
Substituting the definition of $\psi^n_k = q^n_k - \tilde{q}_k$ gives
\begin{equation}
    \Delta E^n = \sum_k{\dint{\tfrac{1}{2}\nu\left(\psi_k^n\right)^2}_{\tilde{T}}\bigg|_{x_{k+1}} - \dint{\tfrac{1}{2}\nu\left(q^{n}_k - \tilde{q}_{k-1}\right)^2}_{\tilde{T}}\bigg|_{x_k} - \dint{\tfrac{1}{2}\left(\psi_k^n\right)^2}_{\tilde{\Omega}}\bigg|_{t^{n-1}}}.
\end{equation}
We then apply Lemma 4.2, giving
\begin{equation}
\Delta E^n = \sum_k{\dint{\tfrac{1}{2}\left(\psi_k^n\right)^2}_{\tilde{\Omega}}\bigg|_{t^{n}} 
    -\dint{\nu\left(\tfrac{1}{2}\left(\psi_k^n\right)^2 + \psi\left(\tilde{q}_{k-1} - q^{n}_k\right) + \tfrac{1}{2}\left(\tilde{q}_{k-1} - q^{n}_k\right)^2\right)}_{\tilde{T}}\bigg|_{x_k}
    },
\end{equation}
expanding $\psi^n_k = q^n_k - \tilde{q}_k$ in the second term yields
\begin{equation}
    \Delta E^n = \sum_k{\dint{\tfrac{1}{2}\left(\psi_k^n\right)^2}_{\tilde{\Omega}}\bigg|_{t^{n}} 
    -\dint{\tfrac{1}{2}\nu\left(\tilde{q}_k - \tilde{q}_{k-1}\right)^2}_{\tilde{T}}\bigg|_{x_k}
    },
\end{equation}
as required.
\end{proof}

This analysis can be easily be extended to non-periodic boundary conditions by treating the leftmost and rightmost element interfaces separately.

\subsection{Stability of the eigenvalue problem}

We now show how to use Theorem 1 to prove numerical stability for a given polynomial order and CFL number. To do this we express the nodal values of the difference between the corrector and the predictor at $t^n$, $\psi_k^n|_{t^n}$, as a linear combination of the nodal values of the jump in the predictor at the left element boundary $x_k$, $\left(\tilde{q}_{k-1} - \tilde{q}_k\right)|_{x_k}$. Let $\boldsymbol{\gamma}_k$ and $\boldsymbol{\beta}_k$ be the $(p+1)$ dimension vectors containing these nodal values, explicitly
\begin{equation}
    \left(\gamma_k\right)_i = \psi^n_k\left(\xi=\xi_i, \tau=1\right),
\end{equation} 
\begin{equation}
    \left(\beta_k\right)_i = \tilde{q}_{k-1}\left(\xi=1, \tau=\tau_i\right) - \tilde{q}_{k}\left(\xi=-1, \tau=\tau_i\right),
\end{equation} 
where $\xi_i$ and $\tau_i$ are the 1D GLL quadrature points in space and time. From \eqref{eq:pred-corr-diff-1D-2}, $\psi_k$ depends linearly on $\left(\tilde{q}_{k-1} - \tilde{q}_k\right)|_{x=x_k}$, therefore $\boldsymbol{\gamma}_k$ depends linearly on $\boldsymbol{\beta}_k$ and we can construct a matrix $A$ such that
\begin{equation}
\boldsymbol{\gamma}_k = \sqrt{\nu} A \boldsymbol{\beta}_k,
\end{equation}
note we have taken the factor $\nu$ outside of $A$ to simplify the proof below.

To construct A we first re-write \eqref{eq:pred-corr-diff-1D-2} in matrix form. Let $M$ be the $(p+1)^2\times(p+1)^2$ matrix such 
\begin{equation}
    M_{ij} = \dint{\phi_i, \pderiv{\phi_j}{\tau}}_{\tilde{\Omega} \times \tilde{T}} + \dint{\phi_i, \nu \pderiv{\phi_j}{\xi}}_{\tilde{\Omega} \times \tilde{T}} + \dint{\phi_i, \phi_j}_{\tilde{\Omega}}\bigg|_{t^{n-1}} - \dint{\phi_i, \nu \phi_j}_{\tilde{T}}\bigg|^{x_{k+1}}_{x_k} 
    \label{eq:M-pred-corr-diff-1D-2},
\end{equation}
where $\phi_i$ are the space-time basis functions,
and $L$ be the $(p+1)^2 \times (p+1)$ matrix
\begin{equation}
    L_{ij} = \dint{\phi_i, l_j(\tau)}_{\tilde{T}}\bigg|_{\xi=-1},
\end{equation}
where $l_j$ is the Lagrange polynomial interpolating the $jth$ GLL quadrature point. Let $\boldsymbol{\psi}_k$ be the vector containing the nodal values of $\psi_k$, therefore \eqref{eq:pred-corr-diff-1D-2} can be written 
\begin{equation}
    M\boldsymbol{\psi}_k = \nu L\boldsymbol{\beta}_k.
\end{equation}
The vector $\boldsymbol{\gamma}_k$ contains the nodal values of $\psi_k^n$ at the $\tau=1$ boundary, let $S$ be the $(p+1) \times (p+1)^2$ matrix which selects these values. Explicitly,
\begin{equation}
    S_{ij} = \phi_j\left(\xi=\xi_i, \tau=1
\right).
\end{equation}
Therefore $\boldsymbol{\gamma}_k = S \boldsymbol{\psi}_k$ and
\begin{equation}
    \boldsymbol{\gamma}_k = \nu SM^{-1}L\boldsymbol{\beta}_k,
\end{equation}
and we construct $A$ as
\begin{equation}
A = \sqrt{\nu} SM^{-1}L.
\end{equation}

The main result of this section is that $\norm{A}_w \leq 1$ implies that our method is numerically stable, where $\norm{\boldsymbol{v}}_w = \sqrt{\sum_i{v_i^2 w_i}}$ is the $L^2$ norm weighted with the 1D GLL quadrature weights $w_i$ and $\norm{A}_w$ is the corresponding vector induced norm. We present this result in the following theorem.


\begin{theorem}
    Under the assumption that $\norm{A}_w \leq 1$ when $\nu \leq 1$, the 1D method \eqref{eq:ref-1D-1}-\eqref{eq:ref-1D-2} is numerically stable for all $\nu \leq 1$, that is
    \begin{equation}
        \dint{\tfrac{1}{2}\left(q^n_k\right)^2}_{\tilde{\Omega}}\bigg|_{t^{n}} - \dint{\tfrac{1}{2}\left(q^{n-1}_k\right)^2}_{\tilde{\Omega}}\bigg|_{t^{n-1}} \leq 0.
    \end{equation}
\end{theorem}
\begin{proof}
    Using theorem 4.1 we obtain
    {
\scriptsize
    \begin{equation}
    \begin{split}
\sum_k{\dint{\tfrac{1}{2}\left(q^{n+1}_k\right)^2}_{\tilde{\Omega}}\bigg|_{t^{n}} - \dint{\tfrac{1}{2}\left(q^{n}_k\right)^2}_{\tilde{\Omega}}\bigg|_{t^{n-1}}} &= \sum_k{
        \dint{\tfrac{1}{2}\left(\psi^n_k\right)^2}_{\tilde{\Omega}}\bigg|_{t^{n}} - \dint{\tfrac{1}{2}\nu\left(\tilde{q}_{k-1} - \tilde{q}_k\right)^2}_{\tilde{T}}\bigg|_{x_k}}\\
&=\frac{1}{2}\sum_k{\sum_i{\left(\left(\gamma_k\right)_i^2 w_i-\nu\left(\beta_k\right)_i^2 w_i\right)}} \\
&= \frac{1}{2}\sum_k{\left(\norm{\boldsymbol{\gamma}_k}_w^2 - \nu \norm{\boldsymbol{\beta}_k}_w^2\right)} \\
&= \frac{\nu}{2}\sum_k{\left(\norm{A\boldsymbol{\beta}_k}_w^2 - \norm{\boldsymbol{\beta}_k}_w^2\right)} \\
&\leq 0.
\end{split}
\end{equation}
}
\end{proof}
We have verified that $\norm{A}_w \leq 1$ for order $p\leq 14$ and variety of CFL numbers $\nu$ up to and including $\nu =1$, these results are shown in figure \ref{fig:A-norm}.

\begin{figure}[!hbtp]\begin{center}
	\includegraphics[width=0.8\textwidth]{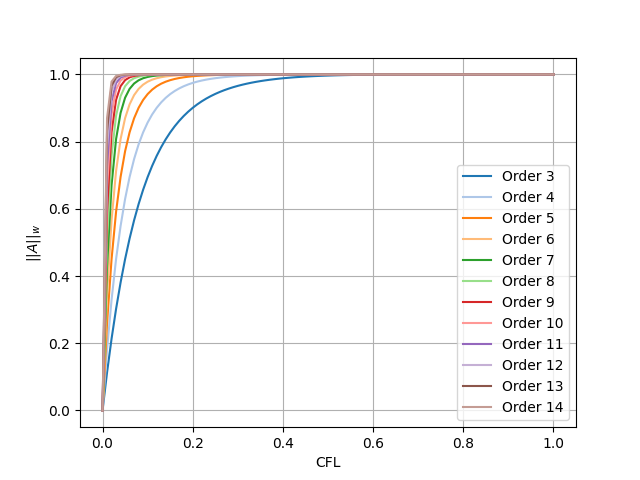}
 \caption{Numerically calculated $\norm{A}_w$.}
  \label{fig:A-norm}
\end{center}
\end{figure}

\section{Von Neumann Stability Analysis}

In this section, we evaluate the stability region of our method on regular Cartesian meshes with constant advection speed $\boldsymbol{u}$ using a numerical von Neumann analysis as in \cite{guthrey2019regionally}. Let $\nu := \frac{u}{\Delta x}$ represent the CFL number in 1D, and $\nu_i := \frac{u_i}{\Delta x_i}$ in higher dimensions. Our goal is to determine the maximum CFL value $\norm{\boldsymbol{\nu}}_2 = \sqrt{\sum_i{\nu_i^2}}$ for which our method remains stable. Note that this analysis also generalises to advection at constant speed on an affine grid, with the CFL instead defined as $\nu_i:=\boldsymbol{u}\cdot \nabla \xi_i$. 


In 1 spatial dimension, we conducted a von Neumann analysis using $100$ evenly sampled wave numbers for a range of different CFLs. Figure \ref{fig:amplification-1D-minus-1} shows the amplification factor minus 1 for order 3-14. A sharp transition is observed at a CFL value of 1. Notably, our von Neumann analysis becomes unstable in double precision for orders greater than 11. However, we verified stability for orders 12, 13, and 14 by increasing the precision.


\begin{figure}[!hbtp]\begin{center}
	\includegraphics[width=0.8\textwidth]{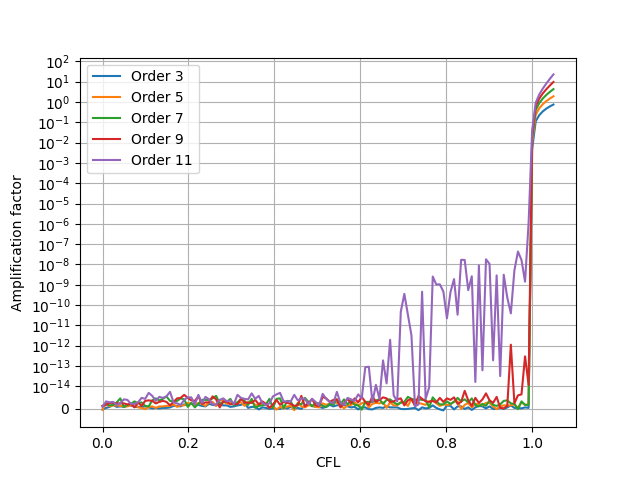}
 \caption{Amplification factor minus 1 for the 1D scheme.}
  \label{fig:amplification-1D-minus-1}
\end{center}
\end{figure}


In higher dimensions, we are interested in the maximum $\norm{\boldsymbol{\nu}}$ for which our method remains stable. To simplify the problem, we note that the most challenging case for a given $\norm{\boldsymbol{\nu}}$ occurs when $\nu_1 = \nu_2$. Therefore, in 2 spatial dimensions we perform our von Neumann analysis with $\nu_1 = \nu_2$ and $100^2$ evenly sampled wave Numbers. Figure \ref{fig:amplification-2D-1} shows the amplification factor minus 1. The scheme is stable for CFL values of around $0.7 \approx 1 / \sqrt{2}$. 


\begin{figure}[!hbtp]\begin{center}
	\includegraphics[width=0.8\textwidth]{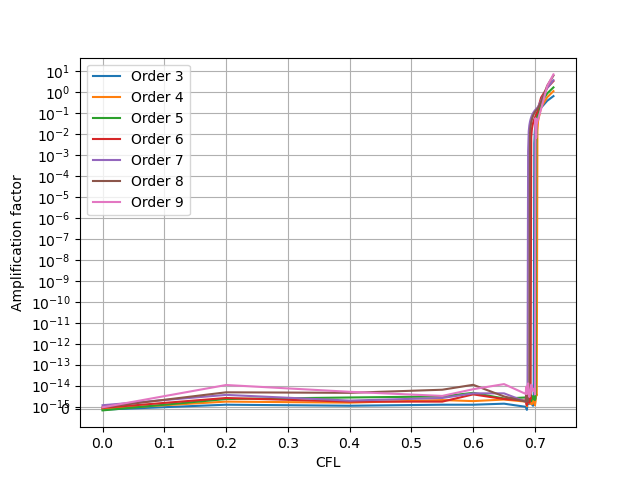}
 \caption{Amplification factor minus 1 for the 2D scheme.}
  \label{fig:amplification-2D-1}
\end{center}
\end{figure}

Similarly in 3 spatial dimensions, we conduct our von Neumann analysis with $\nu_1 = \nu_2 = \nu_3$ and $50^3$ evenly sampled wave Numbers. Figure \ref{fig:amplification-3D-1} shows the amplification factor minus 1. The scheme is stable for CFL values of around $0.58 \approx 1 / \sqrt{3}$. 


\begin{figure}[!hbtp]\begin{center}
	\includegraphics[width=0.8\textwidth]{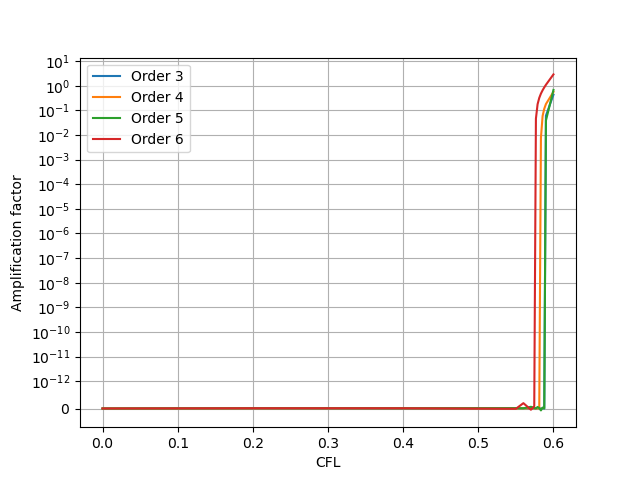}
 \caption{Amplification factor minus 1 for the 3D scheme.}
  \label{fig:amplification-3D-1}
\end{center}
\end{figure}

\section{Numerical results}

In this section we verify the convergence of our method with 2D linear advection, 3D linear advection, 2D Burgers equation, and cubed sphere advection test cases. 

\subsection{2D linear advection}

We verify the convergence of our 2D method by simulating a simple advection test case in the domain $[0,2]^2$ with periodic boundaries. The velocity field is $\boldsymbol{u} = [1, 1]$, and the initial condition is given by
\begin{equation}
	q(\vb{x},0) = \sin{\left(16\pi x\right)}\sin{\left(16\pi y\right)}.
\end{equation}
The simulation runs until $t=2$ with an element CFL $\norm{\boldsymbol{\nu}}_2 = 0.6$. Figure \ref{fig:convergence-advection-2D} shows the relative $L^2$ errors, demonstrating that our method converges at the quasi-optimal rate of $p+\frac{1}{2}$. 

\begin{figure}[!hbtp]\begin{center}
	\includegraphics[width=0.8\textwidth]{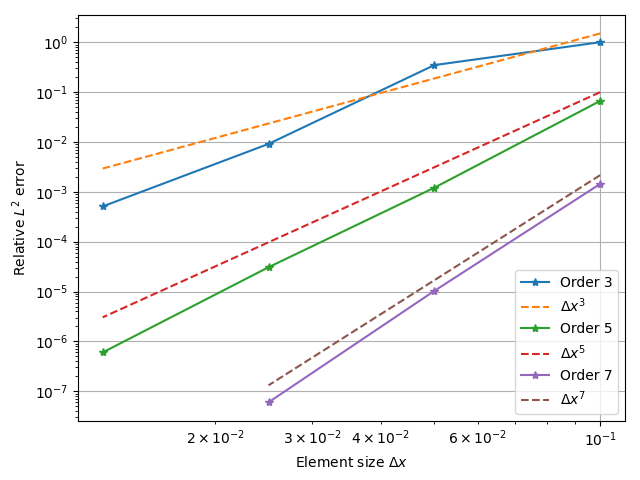}
 \caption{Relative $L^2$ errors for 2D linear advection.}
  \label{fig:convergence-advection-2D}
\end{center}
\end{figure}

\subsection{3D linear advection}

Next, we verify the convergence of our 3D method by simulating a similar advection test case in the domain $[0,2]^3$ with periodic boundaries. The velocity field is $\boldsymbol{u} = [1, 1, 1]$, and the initial condition is given by
\begin{equation}
	q(\vb{x},0) = \sin{\left(2\pi x\right)}\sin{\left(2\pi y\right)}\sin{\left(2\pi z\right)}.
\end{equation}
The simulation runs until $t=2$ with an element CFL $\norm{\boldsymbol{\nu}}_2 = 0.55$. Figure \ref{fig:convergence-advection-3D} shows the relative $L^2$ errors, demonstrating that our method converges at the quasi-optimal rate of $p+\frac{1}{2}$. 

\begin{figure}[!hbtp]\begin{center}
	\includegraphics[width=0.8\textwidth]{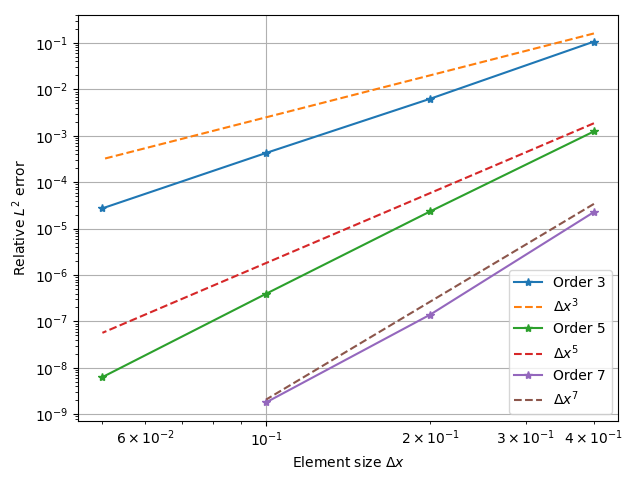}
 \caption{Relative $L^2$ errors for 3D linear advection.}
  \label{fig:convergence-advection-3D}
\end{center}
\end{figure}

\subsection{2D Burger's equation}

We extend our method to the Burgers equation by using an entropy stable DG-SEM spatial discretisation. The simulation is performed on the domain $[0,2\pi]^2$ with the initial condition
\begin{equation}
	q(\vb{x},0) = \frac{1}{4} \left(1 - \cos\left(x\right)\right) \left(1 - \cos\left(y\right)\right)\end{equation}
until $t=0.4$, stopping shortly before a shock develops. The simulation is run with an element CFL $\norm{\boldsymbol{\nu}}_2 = 0.6$. We compute an exact solution using the method of characteristics and use this to evaluate the error. Figure \ref{fig:convergence-burgers-2D} shows the relative $L^2$ error, again demonstrating that that our method converges at the quasi-optimal rate of $p+\frac{1}{2}$. 

\begin{figure}[!hbtp]\begin{center}
	\includegraphics[width=0.8\textwidth]{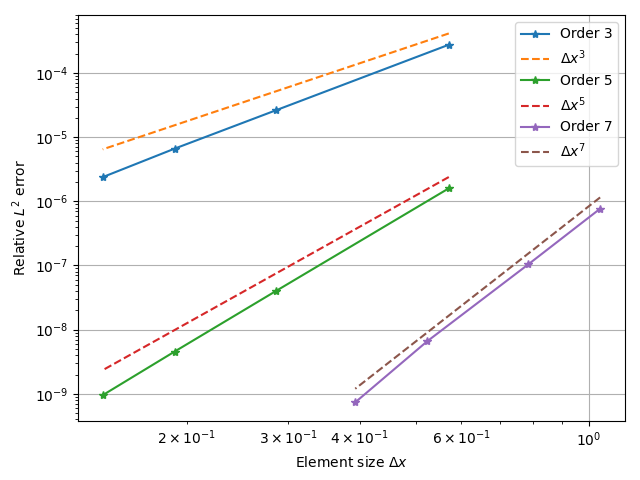}
 \caption{Relative $L^2$ errors for 2D Burger's equation.}
  \label{fig:convergence-burgers-2D}
\end{center}
\end{figure}






\subsection{Cubed sphere advection}

We have also evaluated our scheme on a cubed sphere using the standard test case from \cite{lauritzen2012standard}.This case advects two Gaussian hills advected with a divergence-free flow, causing them to roll up (Figure \ref{fig:flow-cubed-sphere}), before unrolling and returning to their initial state. Simulations are performed with an element CFL of $\norm{\boldsymbol{\nu}}_2 = 0.65$. The $L^2$ error is computed as the difference between the initial and final states and is reported in Figure \ref{fig:convergence-cubed-sphere}. Interestingly, for this more complex test on curvilinear geometry, our scheme achieves a convergence rate that exceeds the optimal order of $p+1$.

\begin{figure}[!hbtp]\begin{center}
	\includegraphics[width=0.8\textwidth]{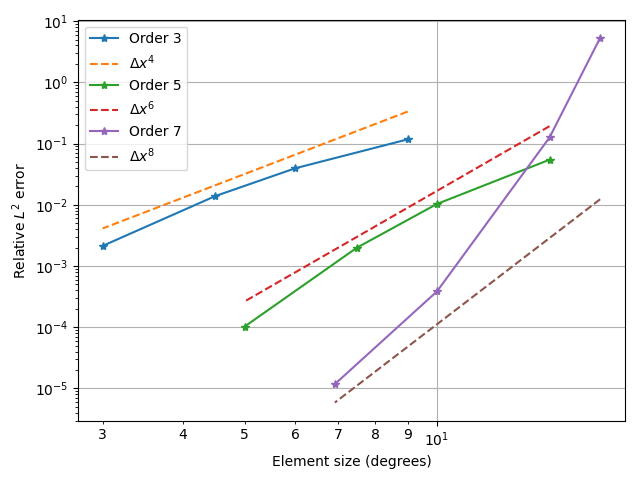}
 \caption{Relative $L^2$ errors for 2D cubed sphere advection.}
  \label{fig:flow-cubed-sphere}
\end{center}
\end{figure}

\begin{figure}[!hbtp]\begin{center}
	\includegraphics[width=\textwidth]{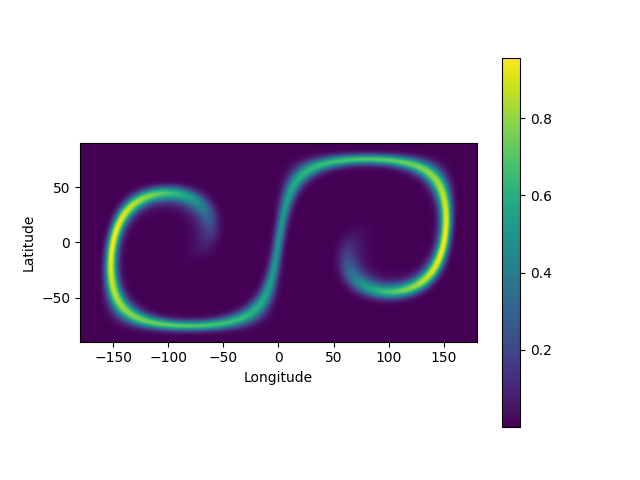}
 \caption{Cubed sphere advection test case at $t=2.5$ with $6\times 13^2$ 7th order elements.}
  \label{fig:convergence-cubed-sphere}
\end{center}
\end{figure}

\section{Conclusion}

In this paper, we introduced a novel locally implicit DG method for transport problems that employs a series of element-local implicit problems, enabling a maximum stable time step that is independent of the polynomial order. This approach significantly increases the maximum stable time step compared to the standard ADER-DG method, thereby reducing the number of communication steps required for simulations and enhancing scalability. We rigorously proved stability in 1D for $\nu \leq 1$, and in 2D and 3D, we derived the maximum stable time step through a semi-analytical von Neumann analysis. The method was validated through a series of linear and nonlinear tests, demonstrating convergence at quasi-optimal order.

Future work will focus on integrating this transport method into numerical solvers for the Euler equations, as well as extending this method to more general systems of hyperbolic problems.

\section*{Declaration of competing interest}

The authors declare that they have no known competing financial interests or personal relationships that could have appeared to influence the work reported in this paper.

\section*{Acknowledgments}
Kieran Ricardo would like to acknowledge the Australian Government through the Australian Government Research Training Program (RTP) Scholarship, and the Bureau of Meteorology through research contract KR2326. The authors would also like to thank David Lee for his insightful comments and discussions.


\clearpage

\printbibliography

\end{document}